\documentclass{amsart}
\usepackage{amsfonts}
\usepackage[alphabetic]{amsrefs} 

\newtheorem{theorem}{Theorem}[section]
\newtheorem{lemma}[theorem]{Lemma}
\theoremstyle{definition}
\newtheorem{definition}[theorem]{Definition}

\theoremstyle{remark}
\newtheorem{remark}[theorem]{Remark}
\numberwithin{equation}{section}
\theoremstyle{plain}

\newtheorem{corollary}[theorem]{Corollary}

\newtheorem{proposition}[theorem]{Proposition}

\begin{document}
\title {A warped product version of the  Cheeger-Gromoll splitting theorem }

\author{William Wylie}
\address{215 Carnegie Building\\
Dept. of Math, Syracuse University\\
Syracuse, NY, 13244.}
\email{wwylie@syr.edu}
\urladdr{https://wwylie.expressions.syr.edu}

\begin{abstract}
We prove a new generalization of the Cheeger-Gromoll splitting theorem where we obtain a warped product splitting under the existence of a line.  The curvature condition in our splitting is a  curvature dimension inequality of the form $CD(0,1)$.  Even though we have to allow warping in our splitting,  we are able to recover topological applications.  In particular, for a smooth compact Riemannian manifold admitting a density which is  $CD(0,1)$, we show that the fundamental group of $M$ is the fundamental group of a compact manifold with nonnegative sectional curvature.  If the space is also locally homogeneous, we obtain that the space also admits a metric of non-negative sectional curvature. Both of these obstructions give many examples of Riemannian metrics which do not admit any smooth density which is $CD(0,1)$. 
    \end{abstract}

\maketitle

\section{Introduction}

The Cheeger-Gromoll splitting theorem  states that a complete manifold with non-negative Ricci curvature that admits a line is isometric to a product metric of the form $ \mathbb{R}\times L $.  A  \emph{line} is a geodesic $\gamma:(-\infty, \infty) \rightarrow M$ which is minimizing between any two points on $\gamma$.    A simple way to construct a space with a line that  is not isometric to a product is to  take the topological product $\mathbb{R} \times L$ with metric $g=dr^2 + g_r$ where $g_r$ with $r\in (-\infty, \infty)$ is a smooth one-parameter family of smooth metrics on $L$.  The splitting theorem implies that any such complete metric $g$ has non-negative Ricci curvature if and only if $g_r = g_0$ and $g_0$ has non-negative Ricci curvature.    

In this paper we give a generalization of the splitting theorem which characterizes a more general class of spaces. If there is a positive function $u(r)$ on $\mathbb{R}$ such that $g = dr^2 + u^2(r) g_0$  for a fixed metric $g_0$ then $g$ is called a \emph{warped product over $\mathbb{R}$}.  Note that such metrics always  contains a line in the $\mathbb{R}$ direction. The curvature condition for our splitting theorem is a curvature dimension inequality, which is a generalization of a lower bound on Ricci curvature.       
\begin{definition}
Let $(M^n,g)$ be a Riemannian manifold and $f$ a smooth real valued function on $M$.  The $N$-dimensional generalized Ricci tensor of the triple $(M,g,f)$ is 
\[ \mathrm{Ric}_f^N  = \mathrm{Ric} +\mathrm{Hess} f - \frac{df \otimes df}{N-n}.  \]
 We say that  $(M,g,f)$ is  $\mathrm{CD}(\lambda, N)$, ($\lambda \in \mathbb{R}, N \in (-\infty, \infty]$) if   $\mathrm{Ric}_f^N \geq \lambda$. 
\end{definition}

If $(M,g)$ has  $\mathrm{Ric} \geq \lambda$ then  taking $f$ to be constant,  we will have $CD(\lambda, N)$ for all $N$.   
    Until recently, the study of curvature dimension inequalities has focused on the cases $N>n$ or $N= \infty$.    However,  there is an emerging body of research for the more general condition where  $N<n$.  The first systematic investigations of the range $N<n$, \cite{Ohta, KolesnikovMilman} appeared almost simultaneously.   \cite{Ohta}  has shown that the $CD(K,N)$ condition $N<0$  is characterized by convexity properties of a relative entropy.  In particular,  this allows one to make sense of the $CD(K,N)$, $N<0$ condition on non-smooth spaces, also see the earlier works  of Ohta-Takatsu\cite{OhtaTakatsu1, OhtaTakatsu2}. Kolesnikov-Milman \cite{KolesnikovMilman} and Milman \cite{Milman1} have also studied isoperimetric, functional, and  concentration properties of spaces satisfying $CD(\lambda, N)$, $N< 1$. Also see Klartag \cite{Klartag}.   For  interesting examples of $CD(K,N)$ densities on the sphere,  see \cite{Milman2}.  As is pointed out by Milman,  the study  of curvature dimension inequalities on Euclidean space with $N<0$ was investigated in the 1970s by Borell \cite{Borell} and Brascamp-Lieb\cite{BrascampLieb}.   We should also caution the reader that our definition of $CD(\lambda, N)$, which matches \cite{Milman1},  is not equivalent to the definition given by Bakry-Emery \cite{BE} in the range $N \in [0,n)$, see \cite[Section 7.5]{Milman1}.   

We will add to these works  by generalizing the splitting theorem to  the $CD(0,N)$ condition where $N \leq 1$.  
The first results for weighted Ricci curvature were proven by Lichnerowicz in \cite{Lichnerowicz1, Lichnerowicz2}.  One of his results (in our notation) is that if  there is a bounded  function $f$ such that $(M,g,f)$ is  $CD(0, \infty)$ then  the Cheeger-Gromoll splitting theorem holds.  Since hyperbolic space admits an unbounded density which is $CD(0, \infty)$,  the assumption that $f$ be bounded is necessary.   Fang-Li-Zhang \cite{FangLiZhang}  also showed that  the splitting theorem holds for  $CD(0, N)$,  $N>n$ with no assumptions on $f$, and improve  Lichnerowicz's result for $CD(0,\infty)$ by  only assuming an upper bound on $f$.   The splitting theorem for non-smooth spaces  in the case $N>1$ has also recently been proven by Gigli \cite{Gigli1, Gigli2}. 

We will show  that  Lichnerowicz's smooth splitting theorem  holds for the weaker  $CD(0,N)$ condition  where $N < 1$.  On the other hand, the theorem is not true when $N=1$ as  there are warped product spaces that admit a bounded function $f$ so that the space is $CD(0,1)$.  Our main result says that these are the only such examples.  

%

\begin{theorem}  \label{Theorem:Splitting} Suppose that a complete Riemannian manifold $(M,g)$  admits a line  and a function $f$ which is bounded above which is $CD(0, 1)$,  then $(M,g)$ is a  warped product over $\mathbb{R}$.  
\end{theorem}

As a corollary of the proof of Theorem \ref{Theorem:Splitting} we obtain the isometric product splitting for  $CD(0, N)$  with  $N < 1$. 

\begin{corollary} \label{Corollary:Splitting} Suppose that $(M,g)$ is a complete Riemannian manifold  that admits a line   and a function $f$ which is bounded above which is $CD(0, N)$  for some $N < 1$, then  $M$ splits isometrically as a product $\mathbb{R}\times L$ and $f$ is a function on $L$ only.
\end{corollary}

\begin{remark} We  actually  prove versions of Theorem \ref{Theorem:Splitting} and Corollary \ref{Corollary:Splitting} that are more general  in two ways. The first is that  we can weaken the upper bound  on $f$ assumption to an integral condition along geodesics that we call $f$-completeness.  This condition turns out to be equivalent to the completeness of a certain weighted  affine connection, see \cite{WylieYeroshkin} for further study in this direction.  Secondly  we have versions where the function $f$ can be replaced with a vector field $X$.  We also prove a version of the splitting theorem for manifolds with boundary.  We delay discussing these results until Sections 5 and 6. 
\end{remark}

By applying the Cheeger-Gromoll splitting theorem iteratively one can show that a  complete non-compact manifold of non-negative Ricci curvature is isometric to a product of a Euclidean space and a space with no lines.  We also obtain a sharp structure theorem for spaces which are $CD(0,1)$ with $f$ bounded above. We obtain a topological splitting $M = \mathbb{R}^k \times N$, but  the metric is a warped product  $g = dr^2+ u^2(r)\left( g_{\mathbb{R}^{k-1}} + g_N\right)$, where $g_N$ is a metric with no lines.   See Theorem \ref{Theorem:IterativeSplitting} below for the precise statement. 

 Despite the weaker geometric splitting,  we are able to recover the classical applications  to topology and  homogeneous spaces of the Cheeger-Gromoll splitting theorem.   For example, an obvious corollary is that if $(M,g,f)$ is  $CD(0,1)$ and $f$ is bounded above then $M$ has at most two ends, and only one end if there is a point with  $\mathrm{Ric}_f^{1} >0$.  

Another topological result comes from applying the splitting theorem to the universal cover of a compact $(M,g,f)$ which is  $CD(0,1)$.  When equipped with the pullback of $f$ and $g$, the  universal cover will also be $CD(0,1)$ and the potential function will be bounded.    In this situation we obtain a sharp geometric  structure theorem for the universal cover,  see Theorem \ref{Theorem:CompactUniversalCover} for the precise statement.   Using a result of Wilking \cite{Wilking}  along with the arguments of Cheeger-Gromoll we obtain the following statement about the topology of compact $CD(0,1)$ spaces. 
%
%

\begin{theorem} \label{Thm:FundGroup} Suppose that $(M,g,f)$ is  $CD(0,1)$ with $M$ compact. Then, 
\begin{enumerate}
\item  $\pi_1(M)$ is the fundamental group of a compact manifold with nonnegative sectional curvature. 
\item $b_1(M) \leq n$ and $b_1(M)=n$ if and only if $M$ is isometric to a flat manifold and $f$ is constant. 
\item If moreover, there is a point where $\mathrm{Ric}_f^{1} >0$, then $\pi_1(M)$ is finite. 
\end{enumerate}
\end{theorem}


This result gives many examples of compact Riemannian manifolds which do not support any function $f$ which is $CD(0,1)$.   In fact, it is an open question whether there is a topological difference between spaces which are $CD(0, N)$ and spaces of non-negative Ricci curvature.  That is, we have the following question:  if $(M,g,f)$  is  $CD(0,N)$ does $M$ also  support a metric with non-negative Ricci curvature?   \cite[Proposition 3.7]{KennardWylie} implies this is true if $(M,g)$ is  compact and  homogeneous and, in fact,  $g$  must have non-negative Ricci curvature.  
%
%
Using the splitting theorem we also obtain a complete  classification to compact locally homogeneous spaces which are $CD(0,1)$. 

\begin{theorem} \label{Theorem:LocallyHomogeneous} Suppose $(M,g,f)$ is $CD(0,1)$ where $(M,g)$ is a compact locally homogeneous space.  Then $M$ is a flat bundle over a compact locally homogeneous space of non-negative Ricci curvature.  In particular, $M$ admits a (possibly different)  invariant metric of nonnegative sectional curvature. 
\end{theorem}

Although Milman \cite{Milman1} has obtained information about spaces which are $CD(0, N)$  with  $N < 1$, Theorem \ref{Theorem:Splitting} appears novel as it seems to be the first result in the literature for the case $N=1$. 
The main new ingredient of the proof of our splitting theorem is a new Bochner type formula  which we use to obtain a new Laplacian comparison theorem for a  $CD(0 ,1)$ space.  Our Bochner formula generalizes the classical one for Ricci curvature in a different way than the Bochner formulas of Lichnerowicz \cite{Lichnerowicz1,  Lichnerowicz2} Bakry-Emery \cite{BE, Bakry}, and Ohta \cite{Ohta} for $CD(0, N)$  in the $N=\infty$, $N>n$, and $N<0$ cases respectively.  Under the $CD(0, 1)$ assumption, our formula only applies to distance functions, but it  gives a philosophical connection between Bakry-Emery's definition of curvature dimension and the results in \cite{Milman1, KolesnikovMilman}.   Further applications of the Bochner formula are developed in \cite{WylieYeroshkin}. 

One way in which to summarize our results is to say that  $N=1$ is a critical parameter  for the splitting theorem where the isometric  splitting theorem fails but  is replaced by the weaker warped product splitting.   A natural question is whether  warped product splitting holds for $N \in (1,n)$.   Our methods do not seem to say anything in this case.  

After the completion of this paper, some similar rigidity theorems for Bakry-Emery Ricci tensors for Lorentzian manifolds have also been established in \cite{WoolgarWylie}.  The same limited loss of rigidity from an isometric product to a warped product is also found in that case. 

We would also like to point out that  the intuition that led us to consider that $N=1$ might be a critical dimension, came from recent work of the author that defines a notion of sectional curvature for manifolds with density \cite{Wylie}.  In that work, we develop a notion of weighted sectional curvature,  which we called $\overline{\sec}_f$.  It comes up  from considering modifying the radial curvature equation applied to  Jacobi fields.  The average of  curvatures $\overline{\sec}_f$ over an orthonormal basis is  $\mathrm{Ric}_f^{1}$ in the same way that the sectional curvatures average to the Ricci curvature.  Some of the examples in the next section arise in \cite{KennardWylie} in the context of studying weighted sectional curvatures and our new Bochner formula can be derived from tracing some of the equations in \cite{Wylie}. 


\medskip
\noindent\emph{Acknowledgements.}  This work was supported by a grant from the Simons Foundation (\#355608, William Wylie). We also thank the referee for a careful reading of the manuscript.

\section {Twisted and Warped Products over $\mathbb{R}$}

In this section we discuss the examples that arise in our splitting theorem and also show that Lichneorwicz splitting theorem does not hold for $CD(0,1)$.  We have to initially consider spaces which are slightly more general than a warped product.    Let $(L,h_L)$ be an $(n-1)$-dimensional Riemannian manifold,  let $M = \mathbb{R} \times L$, and let $\psi(r,p)$ be an arbitrary real valued function on $M$.  A  \emph{twisted product}  metric  over $\mathbb{R}$  is a metric $g_M$ of the form 
\[ g_M = dr^2 + e^{\frac{2\psi}{n-1}} h_L \qquad r\in (-\infty, \infty).\]
Twisted products always contain a line given by the geodesic $\gamma(t) = (t ,x_0)$, where $x_0$ is a fixed point in $L$.    If $\psi$ is a function of $r$ only, then the metric $g_M$ is called a  \emph{warped product}.  

 The connection and Ricci tensor of a twisted product, can be found for example as a special case of the equations in, \cite{FGKU}.
 
 \begin{proposition} \label{Proposition:Computation} Let $g_M = dr^2 + e^{\frac{2\psi}{n-1}} h_L$ and let $U$ and $V$ be vector fields on $L$. Then the Riemannian connection of $g_M$ is given by 
\begin{eqnarray*}
\nabla_{ \frac{\partial}{\partial r}} V &=& \frac{1}{(n-1)}\frac{\partial \psi}{\partial r} V \\
\nabla_{U} V &=& \nabla_U^L V + \frac{1}{n-1}\left( D_U(\psi) V + D_V(\psi) U - g_M(U,V) \nabla \psi \right). 
\end{eqnarray*}
The Ricci tensor is given by 
\begin{eqnarray*}
\mathrm{Ric}\left( \frac{\partial}{\partial r}, \frac{\partial}{\partial r} \right) &=& - \frac{\partial^2 \psi}{\partial r^2} - \frac{1}{n-1} \left(  \frac{\partial \psi}{\partial r} \right)^2\\
\mathrm{Ric}\left( \frac{\partial}{\partial r}, V \right)&=& \frac{(2-n)}{n-1} D_V D_{\frac{\partial}{\partial r}} \psi\\
\mathrm{Ric}\left( U, V \right)&=& \mathrm{Ric}_{h_L}\left( U,V\right) + \left( \frac{1}{n-1}\right) \mathrm{Hess}\psi \left(U,V\right) + D_UD_V(\psi) - D_{\nabla^L_U V}(\psi) \\
&& + \frac{1}{n-1}D_U(\psi) D_V(\psi) -  \frac{1}{n-1}\left( \Delta \psi + \frac{|\nabla \psi|^2}{n-1} \right)g_M(U,V). \\
\end{eqnarray*}
\end{proposition}

A natural choice for the potential function $f$ is $f = \psi$,  since then it follows from the equations above that $\mathrm{Ric}^1_{f} \left( \frac{\partial}{\partial r}, \frac{\partial}{\partial r} \right)  = 0$.   On the other hand,  we can also  show that if the potential  $f = \psi$ is  $CD(0,1)$  then  the metric can be written as a warped product.  

 \begin{proposition} \label{Prop:WarpedSplittingGradient} Suppose that $(M,g_M,f)$ is $CD(0,1)$ with the metric of the form  $g_M = dr^2 + e^{\frac{2f}{n-1}} h_L$, then $f = \phi(r) + f_L(x)$, where $\phi: \mathbb{R} \rightarrow \mathbb{R}$ and $f_L: L \rightarrow \mathbb{R}$.  In particular,  the metric $g_M$ is a warped product of the form $g_M = dr^2 + e^{\frac{2\phi(r)}{n-1}}g_{L}$ where $g_L = e^{\frac{2 f_L(x)}{n-1}}h_L$. 
\end{proposition}

\begin{proof}

Since $\mathrm{Ric}^1_f \left( \frac{\partial}{\partial r}, \frac{\partial}{\partial r} \right)  = 0$  and  $\mathrm{Ric}^1_f \geq 0$ we must have that that $\mathrm{Ric}^1_f \left( \frac{\partial}{\partial r}, V \right) = 0$ for all $V \perp \frac{\partial}{\partial r}$. 

Fix a point $x$ in $L$, let $\frac{\partial}{\partial y^i}$, $i=1, \dots, n-1$ be an orthonormal basis of local coordinates around $x$ in the $h_L$ metric.  Write $\nabla f = a(r,y) \frac{\partial}{\partial r} +  b_i(r,y) \frac{\partial}{\partial y^i} $, then $a= \frac{\partial f}{\partial r}$ and $b_i   = e^{\frac{-2f}{n-1}} \frac{\partial f}{\partial y^i}$.

 Then we have 
\begin{eqnarray*}
\mathrm{Hess} f\left(\frac{\partial}{\partial r}, \frac{\partial}{\partial y^k} \right ) &=& g_M\left( \nabla_{\frac{\partial}{\partial r}} \nabla f , \frac{\partial}{\partial y^k} \right) \\
&=& \sum_{i=1}^{n-1} g_M\left( \nabla_{\frac{\partial}{\partial r}}\left( e^{\frac{-2f}{n-1}} \frac{\partial f}{\partial y^i} \frac{\partial}{\partial y^i} \right),\frac{\partial}{\partial y^k}\right) \\
&=& \frac{\partial }{\partial r } \left(e^{\frac{-2f}{n-1}} \frac{\partial f}{\partial y^k}  \right) e^{\frac{2f}{n-1}} + \frac{1}{n-1} \frac{\partial f}{\partial y^k}  \frac{\partial f}{\partial r} \\
&=&  \frac{\partial }{\partial r } \frac{\partial }{\partial y^k} (f)  - \frac{1}{n-1} \frac{\partial f}{\partial y^k}  \frac{\partial f}{\partial r}.
\end{eqnarray*}

This combined with Proposition \ref{Proposition:Computation} implies that
\begin{eqnarray*}
0 &=& \mathrm{Ric}_f^{1}\left(\frac{\partial}{\partial r},\frac{\partial}{\partial y^k} \right)\\
&=& \frac{2-n}{n-1}\frac{\partial }{\partial y^k }  \frac{\partial }{\partial r}   ( f)
 +  \frac{\partial }{\partial r } \frac{\partial }{\partial y^k} (f)  - \frac{1}{n-1} \frac{\partial f}{\partial y^k}  \frac{\partial f}{\partial r} +  \frac{1}{n-1} \frac{\partial f}{\partial y^k}  \frac{\partial f}{\partial r} \\
 & =& \left(1 + \frac{2-n}{n-1} \right) \frac{\partial }{\partial r}  \frac{\partial }{\partial y^k } ( f)  \\
 &=&  \frac{1}{n-1} \frac{\partial }{\partial r}  \frac{\partial }{\partial y^k } ( f) = \frac{1}{n-1}\frac{\partial }{\partial y^k }  \frac{\partial }{\partial r}   ( f)
\end{eqnarray*}
This implies that $\frac{\partial f}{\partial r}$ is constant in directions tangent to $L$ and $\frac{\partial f}{\partial y^k}$ is constant in the $r$ direction.  As in \cite[Theorem 1]{FGKU} this implies that $f = \phi(r) + f_L$ where $f_L:L \rightarrow \mathbb{R}$.  Then 
\begin{eqnarray*}
g_M = dr^2 + e^{\frac{2f}{n-1}}h_L =   dr^2 + e^{\frac{2\phi(r)}{n-1}}\left( e^{\frac{2 f_L(x)}{n-1}} h_L\right)
\end{eqnarray*}
which gives the result. 

\end{proof}

The  triples $(M,g,f)$  of the form given by the conclusion of  Proposition \ref{Prop:WarpedSplittingGradient} are exactly the spaces that arise in our splitting theorem.  To aid our exposition we will call these triples \emph{split spaces}.  That is $(M,g_M,f)$ is a split space if $M$ is diffeomorphic to $\mathbb{R} \times L$, $f = \phi(r) + f_L(x)$, where $\phi: \mathbb{R} \rightarrow \mathbb{R}$ and $f_L: L \rightarrow \mathbb{R}$, and $g_M = dr^2 + e^{\frac{2\phi(r)}{n-1}}g_{L}$ for a fixed metric $g_L$ on $L$.   From the calculations above we have that $\mathrm{Ric}_f^1 \left( \frac{\partial}{\partial r}, \cdot \right) = 0$ for any split space. 

 In order for a split space to be $CD(0,1)$ we need an additional curvature assumption for the triple $(L, g_L, f_L)$.  In considering what this condition should be,  note that it is not true that the isometric product of spaces which are $CD(0, 1)$ are $CD(0,1)$.  This is because the definition of the curvature dimension condition depends on the dimension of the manifold.  In fact,  the product metric $L^{n-k} \times \mathbb{R}^k$ with a  potential function $f$ defined on $L$ admits $CD(0,1)$ if and only if  $(L^{n-k}, g_L, f)$ is  $CD(0, 1-k)$.    This motivates the following proposition. 

\begin{proposition} \label{Proposition:GradientExample}
A split space $(M,g_M, f)$  is $CD(0,1)$ if and only if 
\[  (\mathrm{Ric}_{g_L})^{0}_{f_L}  \geq \sup_r  \left( \frac{1}{n-1} \frac{\partial^2 \phi} {\partial r^2}  e^{\frac{2 \phi}{n-1}} \right) g_L \]
In particular,  if $(M,g_M, f)$ is $CD(0,1)$ then $(L, g_L, f_L)$ is $CD(0,0)$.
\end{proposition}
\begin{proof}
From Proposition  \ref{Prop:WarpedSplittingGradient} we already see that $\mathrm{Ric}^1_f \left( \frac{\partial}{\partial r}, Y \right)  = 0$ for all vector fields $Y$, so we just need to consider $\mathrm{Ric}^1_f \left( U,V \right) $ for $U,V \perp  \frac{\partial}{\partial r}$.  From Proposition \ref{Proposition:Computation} we have that 
\begin{eqnarray*}
\mathrm{Hess} \phi(U,V) &=& \frac{1}{n-1} \left( \frac{\partial \phi}{\partial r}\right)^2 g_M(U,V) \\
\Delta \phi &=&  \frac{\partial^2 \phi} {\partial r^2} + \left( \frac{\partial \phi}{\partial r}\right)^2 
\end{eqnarray*}
and thus 
\begin{eqnarray*}
\mathrm{Ric}\left( U, V \right)&=& \mathrm{Ric}_{g_L}\left( U,V\right) - \frac{1}{n-1}\left(  \frac{\partial^2 \phi} {\partial r^2} + \left( \frac{\partial \phi}{\partial r}\right)^2\right)g_M(U,V). 
\end{eqnarray*}
Moreover, 
\[ \mathrm{Hess} f (U,V) = \mathrm{Hess} \phi (U,V) + \mathrm{Hess}^L f_L (U,V) =  \frac{1}{n-1} \left( \frac{\partial \phi}{\partial r}\right)^2 g_M(U,V) + \mathrm{Hess}^L f_L (U,V). \]
So 
\begin{eqnarray*}
\mathrm{Ric}_f^{1}(U,V) &=& \left(\mathrm{Ric}_{g_L}\right)_{f_L}^0(U,V) -  \frac{1}{n-1}\left(  \frac{\partial^2 \phi} {\partial r^2} \right) g_M(U,V).
\end{eqnarray*}
Which gives the first part of the result.  

Now suppose that $(L, g_L, f_L)$ is not $CD(0, 0)$, then there is a  constant $a>0$ such that 
 \[\frac{1}{n-1} \frac{\partial^2 \phi} {\partial r^2} \leq -a e^{\frac{-2 \phi}{n-1}}\]
 Letting $y = \phi/(n-1)$ we have $y'' \leq -a e^{-2y}$.  Solutions to this inequality can be bounded above by the appropriate solutions to the equation $v'' = -ae^{-2v}$. This equation can be solved explicitly and we can see that all solutions $v$ go to $-\infty$ in finite time, but this contradicts that $\phi$ is defined for all $r$. 
 \end{proof}

Now we can construct the examples of spaces with bounded $f$  and containing a line which do not split as products but are $CD(0,1)$.  These spaces show that Lichnerowicz's splitting theorem does not hold for the $CD(0,1)$ condition.  
  
  \begin{corollary}  Let $\phi: \mathbb{R} \rightarrow \mathbb{R}$ be a bounded $C^2$ function which has bounded first and second derivatives.  Then there exists $\lambda$ large enough such that  the metric $dt^2 + e^{\frac{2\phi}{n-1}} g_{S^n_\lambda}$ with $f = \phi$   is $CD(0,1)$ where $S^n_{\lambda}$ is the sphere of constant Ricci curvature $\lambda$. 
  \end{corollary}
  \begin{proof}
 In the calculations above we have  $f_L= 1$.  Choose $\lambda$ such that  $ \lambda  \geq \sup_r  \left( \frac{1}{n-1} \frac{\partial^2 \phi} {\partial r^2}  e^{\frac{2 \phi}{n-1}} \right)$, then by Proposition \ref{Proposition:GradientExample}, the desired space is $CD(0,1)$.  \end{proof}

In the next section, we will show that split spaces are the only complete spaces with $f$ bounded above  which are $CD(0,1)$ and contain a line.

%
%
%
%
%

\section{Proof of the splitting theorem }

We now turn our attention to proving the splitting theorem.  The first component  is a Bochner formula. The usual Bochner formula for Ricci curvature is that for a $C^3$ function $h$ we have

\[ \frac{1}{2} \Delta|\nabla h|^2  = |\mathrm{Hess} h|^2 + \mathrm{Ric}(\nabla h, \nabla h) + g(\nabla h, \nabla \Delta h). \]
Using Cauchy-Schwarz on the $|\mathrm{Hess} h|^2$ term and assuming the Ricci curvature bound $\mathrm{Ric} \geq K$  gives 
\[ \frac{1}{2} \Delta   |\nabla  h|^2  \geq \frac{(\Delta h)^2}{n}  + K|\nabla h|^2 + g(\nabla h, \nabla \Delta h). \]

Now let $f$ be a function on $M$, the weighted, or $f$-Laplacian  is $\Delta_f = \Delta - D_{\nabla f}$. Then one has the following formula, \cite{Lichnerowicz1} 

\begin{eqnarray} \frac{1}{2}\Delta_f |\nabla  h|^2 =  |\mathrm{Hess} h|^2 + \mathrm{Ric}^{\infty}_f(\nabla h, \nabla h) + g(\nabla h, \nabla \Delta_f h). \label{eqn:WeightedBochner} \end{eqnarray}

For curvature dimension inequalities of generalized dimension less than $n$  we have the following Bochner type formula.

\begin{lemma} \label{Lem:Bochner} Let $(M^n,g,f)$ be a manifold with density that is $CD(K, n-m)$ for some integer $m = 1, 2, \dots n$.   Suppose that $h$ is a $C^{3}$ function in a neighborhood of a point $p$  such that $\mathrm{Hess} h|_p$ has $m$ non-zero eigenvalues.  Let $v = e^{f/m}$,  then 
\[ \frac{1}{2} v^2 \Delta_f |\nabla h|^2\ \geq  v^2 \frac{(\Delta_f h)^2}{m} + v^2 K |\nabla h|^2 + g(\nabla h, \nabla (v^2 \Delta_f h)) \]
Moreover, equality is achieved if and only if the $m$ non-zero eigenvalues of $\mathrm{Hess} h|_p$ are all equal and $\mathrm{Ric}_f^{n-m}(\nabla h, \nabla h) = K|\nabla h|^2$. 
\end{lemma} 

\begin{proof} 
We start with (\ref{eqn:WeightedBochner}) multiplied by $v^2$,

\begin{eqnarray*} v^2\frac{1}{2}\Delta_f  |\nabla  h|^2 = v^2 |\mathrm{Hess} h|^2 + v^2\mathrm{Ric}^{\infty}_f(\nabla h, \nabla h) + v^2g(\nabla h, \nabla \Delta_f h).   \end{eqnarray*}
Then we have 
\[ g(\nabla h, \nabla(v^2 \Delta_f h)) = v^2g(\nabla h, \nabla \Delta_f h)  + 2 v^2\frac{g(\nabla h, \nabla f)}{m}\Delta_f h  \]
and 
\begin{eqnarray*}
v^2|\mathrm{Hess} h|^2 \geq v^2\frac{(\Delta h)^2}{m} &=& v^2\frac{(\Delta_f h + g(\nabla h, \nabla f))^2}{m}\\
&=&  v^2 \left( \frac{(\Delta_f h)^2}{m} + 2 \frac{g(\nabla h, \nabla f)}{m} \Delta_f h + \frac{g(\nabla h, \nabla f)^2}{m} \right).
\end{eqnarray*}
Combining these three equations gives 
\[ \frac{1}{2} v^2 \Delta_f |\nabla h|^2\ \geq  v^2 \frac{(\Delta_f h)^2}{m} + v^2\mathrm{Ric}_f^{n-m}(\nabla h, \nabla h) + g(\nabla h, \nabla (v^2 \Delta_f h)).\]
Applying $\mathrm{Ric}_f^{n-m} \geq K$ then gives the formula in the lemma.  

If the inequality is an equality then we must have  $\mathrm{Ric}_f^{n-m}(\nabla h, \nabla h) = K|\nabla h|^2$ and $|\mathrm{Hess} h|^2 = \frac{(\Delta h)^2}{m}$,  which implies all of the non-zero eigenvalues of $\mathrm{Hess} h$ are the same. 
\end{proof} 

Note that Lemma  \ref{Lem:Bochner} will apply to any function $h$ when $m = n$, thus it gives a Bochner formula for $CD(K, 0)$.  In this paper, we'll be applying this  to  a (generalized) distance function $r$, i.e. a function such that $|\nabla r|=1$ on an open set where the function $r$ is smooth.  For a distance function,  we have $\nabla_{\nabla r} \nabla r = 0$ implying that $\mathrm{Hess} r$ has at most  $(n-1)$ non-zero eigenvalues and that the integral curves of $r$ are unit speed geodesics.    From the Bochner formula we  derive a new Laplacian comparison  for the distance function  for the condition $CD(0,1)$.  

\begin{theorem} 
Let $(M, g, f) $ satisfy the $CD(0,1)$ condition.   Fix a point $p \in M$ and let $r$ be the distance function to $p$.  Let $q$ be a point such that $r$ is smooth at $q$, and let $\gamma(t)$ be the unique minimal geodesic from $p$ to $q$, parametrized by arc-length.   Then 
\[ (\Delta_f r)(q) \leq  \frac{(n-1)}{ v^2(q) \int_0^{r(q)} v^{-2}(\gamma(t)) dt}, \]
where $v = e^{\frac{f}{n-1}}$. 
\end{theorem}

\begin{remark} Note that when $f$ is constant, we have that $v = c$ for a positive constant and
\[  v^2(q) \int_0^{r(q)} v^{-2}(\gamma(t)) dt = r(q) \]
so we recover the usual Laplacian comparison, $\Delta r \leq \frac{n-1}{r}$ for $\mathrm{Ric} \geq 0$. 
\end{remark}

\begin{proof} 

Apply Lemma  \ref{Lem:Bochner} to $h=r$ to obtain 
\[ \frac{d}{dt}\left( v^2 \Delta_f r \right)  \leq  - v^2 \frac{(\Delta_f r)^2}{n-1}.   \]
If we set $\lambda = \left(v^2 \Delta_f r\right)\circ \gamma$ we have 
\[ \dot{\lambda} \leq  - \frac{\lambda^2}{v^2(n-1)}\]
which is a Ricatti equation  that was also used in \cite{Wylie}.    For any sufficiently small $\varepsilon$, we have 
\begin{eqnarray*}
\dot{\lambda} &\leq&  - \frac{\lambda^2}{v^2(n-1)}\\
(n-1) \int_{\varepsilon}^{r(q)} \frac{\dot{\lambda}}{\lambda^2}  dt&\leq& - \int_{\varepsilon}^{r(q)} v^{-2}(\gamma(t)) dt \\
(n-1) \left( -\lambda^{-1} (r(q)) + \lambda^{-1} (\varepsilon)\right) &\leq& - \int_{\varepsilon}^{r(q)} v^{-2}(\gamma(t)) dt. \\
\end{eqnarray*}
Since $\lambda(\varepsilon) \rightarrow \infty$ as $\varepsilon \rightarrow 0$ we have 
\begin{eqnarray*}
(n-1) \left( -\lambda^{-1} (r(q)) \right)&\leq& - \int_0^{r(q)} v^{-2}(\gamma(t)) dt \\
\lambda(r(q)) &\leq& \frac{(n-1)}{\int_0^{r(q)} v^{-2}(\gamma(t)) dt}.
\end{eqnarray*}
This implies the result by the definition of $\lambda$. 
\end{proof}


The proof of our splitting theorem follows the classical argument using Busemann functions. Given a non-compact manifold $M$ and a ray $\gamma$ we define the Busemann function  to $\gamma$ to be the function $b^{\gamma}(x) = \lim_{t \rightarrow \infty} (t - d(x, \gamma(t)))$.   $b^{\gamma}$ is Lipschitz with Lipschitz constant $1$ and is thus differentiable almost everywhere.    We want to show that  when we have a $CD(0,1)$ space with $f$ bounded above, then $\Delta_f b^{\gamma} \geq 0$.   

At the points where the Busemann function is not smooth, we interpret $\Delta_f b^{\gamma}$ in the weak sense in terms of barrier functions.  That is,  for a Lipschitz function $h$ we say that $\Delta_f(h) \geq 0$  at a point $x$ if, for every $\varepsilon>0$,  there is a $C^2$  function $h_{\varepsilon}$ defined in a neighborhood of $x$ such that $h_{\varepsilon}(x) = b^{\gamma}(x)$, $h_{\varepsilon} \leq b^{\gamma}$ in a neighborhood of $x$, and $\Delta_f(h_{\varepsilon}) \geq - \varepsilon$.     The notion that a function have $\Delta_f h \leq 0$ is defined similarly.  We call the functions $h_{\varepsilon}$ \emph{barrier functions}. 

\begin{lemma} \label{Lemma:Busemann} Suppose that $(M,g,f)$ is $CD(0,1)$  and $f$ is bounded above,  then $ \Delta_f( b^{\gamma}) \geq 0 $.
\end{lemma}

\begin{proof}
Let $x \in M$, we construct the barrier functions for $b^{\gamma}$ in the standard way.  That is, let $t_i \rightarrow \infty$ and let $\sigma_i$ be minimal geodesics from $x$ to $\gamma(t_i)$, the sequence $\sigma'(0)$ sub-converges to some $v \in T_xM$.  Let  $\overline{\gamma}$ be the geodesic with $\overline{\gamma}(0) = x$ and $\overline{\gamma}'(0) = v$.  Then $\overline{\gamma}$ is a ray, called an asymptotic ray to $\gamma$. 

Define $h_t(y) = t - d(y, \overline{\gamma}(t))+ b^{\gamma}(x)$, by the standard arguments in for example \cite{WeiWylie}, $h_t$ is  a smooth barrier function to $b^{\gamma}$ at $x$.  Now we compute 
\begin{eqnarray*}
\Delta_f( h_t) &=& - \Delta_f(d(y, \overline{\gamma}(t)) \geq \frac{ -(n-1)}{v^2(y) \int_0^{d(y, \overline{\gamma}(t))} v^{-2}(\gamma(s)) ds}  
\end{eqnarray*}
By the assumption that $f$ is bounded from above,  we have that the quantity $ \int_0^{d(y, \overline{\gamma}(t))} v^{-2}(\gamma(s)) ds$ goes to $\infty$ as $t \rightarrow \infty$, implying that $\Delta_f(b^{\gamma}) \geq 0$. 
\end{proof}

Aside from using the Bochner formula to control the Laplacian of the distance function and thus the Busemann functions, the other application of the Bochner formula used in the splitting theorem is in classifying constant gradient harmonic functions on spaces with $\mathrm{Ric} \geq 0$ as linear functions in a flat factor.  We get a different rigidity classification for  $CD(0,1)$.

\begin{lemma} \label{Lemma:Rigidity} Suppose that  $(M,g, f)$ is  $CD(0,1)$ where $(M,g)$ is a complete Riemannian manifold.   If  there is a smooth function $r$ on $(M,g)$  such that $|\nabla r|^2=1$  and $\Delta_f(r) = 0$, then  the metric $g$ is a warped product of the form $g = dr^2 + e^{\frac{2\phi(r)}{n-1}}g_{L}$ and $f = \phi(r) + f_L$ where $f_L: L \rightarrow \mathbb{R}$.
\end{lemma}

\begin{proof}
The fact that $M$ splits topologically as $\mathbb{R} \times L$ is a simple consequence of Morse theory and is true whenever one has a smooth function $r$ with $|\nabla r|=1$.  We  can write the metric as   $g = dr^2 + g_r$, where $g_r$ is the metric restricted to a  level set of $r$.   The assumptions imply that we have equality in Lemma \ref{Lem:Bochner}, so $\mathrm{Ric}_f^{1} (\nabla r, \nabla r)= 0$ and $\mathrm{Hess} r = \alpha g_r$ for some function $\alpha$.  But we also have $\Delta r = (n-1) \alpha = g(\nabla f, \nabla r)$ so 
\[ \mathrm{Hess} r =\frac{ g(\nabla f, \nabla r)}{n-1} g_r . \]
This implies that 
\[ L_{\nabla r} \left(e^{\frac{-2f}{n-1}} g_r\right) = 0 \]
which implies that 
$g_r = e^{\frac{2\left( f(r, \cdot) - f(0, \cdot)\right)}{n-1}} g_0$. 
This gives us that the metric is a twisted product $g = dr^2 + e^{\frac{2f}{n-1}} g_L$ where $g_L=e^{ \frac{-2f(0, \cdot)}{n-1}}g_0$ is a fixed metric on $L$.  Proposition \ref{Prop:WarpedSplittingGradient}  then implies a warped product splitting, which completes the proof. 
\end{proof}

Now with the lemmas above we can quickly prove the splitting theorems using the standard arguments involving Busemann functions. 

\begin{proof}[Proof of Theorem \ref{Theorem:Splitting}]

Let $\gamma$ be a line in our space, and let $\gamma^+$ and $\gamma^-$ be the two rays that make up the line $\gamma$. Let $b^{\pm}$ be the corresponding Busemann functions.  From Lemma \ref{Lemma:Busemann} we know that $\Delta_f(b^{\pm}) \geq 0$.  Using the standard arguments in the first part of the proof of, for example,  \cite[Theorem 6.1]{WeiWylie} using  the maximum principle one obtains that $b^{+} = - b^{-}$  thus  $\Delta_f(b^{\pm})= 0$, which implies that  $b^{\pm}$ are both smooth by elliptic regularity.  An additional standard argument then gives that  $|\nabla (b^{\pm})| = 1$ at every point.   From Lemma \ref{Lemma:Rigidity} we obtain the  warped  product splitting.  
\end{proof}

\begin{proof}[Proof of Corollary \ref{Corollary:Splitting}] 
Since $(M,g,f)$ is $CD(0,N)$, it is also $CD(0,1)$ so Theorem \ref{Theorem:Splitting} implies that $g$ is a  warped product, $g= dr^2 + e^{\frac{2\phi}{n-1}} g_L$ and $f = \phi(r) + f_L$.    As we saw in Section 2, we also have that $\mathrm{Ric}_f^{1}( \frac{\partial}{\partial r} ,  \frac{\partial}{\partial r} ) = 0$. Then, 
\begin{eqnarray*}
\mathrm{Ric}_f^{N}\left( \frac{\partial}{\partial r} ,  \frac{\partial}{\partial r} \right) &=& \mathrm{Ric}_f^{1} \left( \frac{\partial}{\partial r} ,  \frac{\partial}{\partial r} \right) + \left( \frac{N-1}{(n-1)(n-N)} \right) \left(\frac{d\phi}{dr}\right)^2  \\
&=&  \left( \frac{N-1}{(n-1)(n-N)} \right) \left(\frac{d\phi}{dr}\right)^2. 
\end{eqnarray*}
Since $N-1<0$ we must have $\frac{d\phi}{dr} = 0$.    This implies that metric $g$ is a product metric, which we can write as $g = dr^2 + g_L$ and   $f$ is a function on $L$ only. \end{proof}

\section{Structure theorem and applications}

Now we consider applying our splitting theorem iteratively.  First note that if we have a space with a line that is $CD(0,N)$ for $N<1$ with $f$ bounded above,  then  we have an isometric product splitting $M = \mathbb{R} \times L$ and $f$ is a function on $L$.    Then, $(L, g_L, f_L)$ is $CD(0, N-1)$ and $f_L$ is bounded above, so if $L$ contains a line  then we can apply the splitting theorem to $L$.  Iterating this argument, one obtains that $M$ is isometric to a product metric of the form $M = \mathbb{R}^k \times L$ and $f$ is a function on $L$ with $(L, g_L, f_L)$ being $CD(0, N-k)$.   A similar argument in the $CD(0,1)$ case yields the following. 

 \begin{theorem} \label{Theorem:IterativeSplitting}
Suppose that $(M,g,f)$ is $CD(0,1)$ and $f$ is bounded above,  then $M$ is diffeomorphic to $\mathbb{R}^k \times L$ and the metric $g$ is of the form
\begin{eqnarray*}
g &=& dr^2 + e^{\frac{2\phi(r)}{n-1}} g_{\mathbb{R}^{k-1}} + e^{\frac{2\phi(r)}{n-1}} g_L 
\end{eqnarray*}
Where $g_{\mathbb{R}^{k-1}}$ denotes the Euclidean metric, $(L, g_L)$ has no lines, $f= \phi(r) + f_L$, and  $(L, g_L, f_L)$ is $CD(0, 1-k)$. \end{theorem}

\begin{proof}
Let $(M, g, f)$ be $CD(0,1)$  with $f$ bounded above  and containing a line.  Then we have $g = dr^2 + e^{\frac{2\phi(r)}{n-1}} g_{L'}$, $f = \phi + f_{L'}$ and by Proposition \ref{Proposition:GradientExample} $(L, g_{L'}, f_{L'})$ is $CD(0,0)$.  Since $f$ is bounded above, so is $f_{L'}$ so we can split $L'$ isometrically as $\mathbb{R}^{k-1} \times L$ with  $f_{L'}=f_L$ is a function on $L$ only and  $(L, g_{L}, f_L)$ is $CD(0, 1-k)$ and $L$ contains no lines.  Then  we obtain the splitting
\begin{align*} g = dr^2 +e^{\frac{2\phi(r)}{n-1}} g_{\mathbb{R}^{k-1}} + e^{\frac{2\phi(r)}{n-1}} g_{L}. \end{align*}
\end{proof}
 
Despite the ease with which we can prove this structure theorem, there is one subtle point that is important for the applications below.  That is,  for a warped product, it is not  true that lines in the fiber $L$ will always lift to lines in $M$ nor that lines  in $M$ always project to lines in $L$.    For a simple example of the later case,    consider Euclidean space  written in polar coordinates $dr^2 + r^2 g_{S^{n-1}}$ and a line that is not through the origin.   We first show that this issue with projections is excluded if we use the fact again that $f$ is bounded above. 

\begin{proposition} \label{Proposition:WPLines}
Consider a warped product metric of the form  $g = dr^2 + v^2(r) g_L$ where $v>0$ is bounded from above. Let $\gamma: (a,b) \rightarrow \mathrm{M}$ be a unit speed minimizing geodesic  in  $M$ and write  $\gamma(s) = (\gamma_1(s), \gamma_2(s))$,  where $\gamma_1$ and $\gamma_2$ are the projections in the factors $\mathbb{R}$ and $L$.    Then 
\begin{enumerate}
\item  $\gamma_2$ is  either constant or its image is a minimizing geodesic in $(L,g_L)$. 
\item If $\gamma_2$ is not constant and $\gamma$ is a line in $M$, then the image of $\gamma_2$ is a line in $L$. 
\end{enumerate}
\end{proposition}

Note that $\gamma_2(s)$ itself will not necessarily be a geodesic because it will not be parametrized with constant speed. 

\begin{proof} 
First we  want to show that the image of $\gamma_2$ is a  length minimizing curve in $g_L$.  To see this, parametrize $\gamma$ such that $\gamma:[0,1] \rightarrow M$, then 
\[ length(\gamma) = |\dot{\gamma}(t) | = \sqrt{ |\dot{\gamma_1}(t)|^2_{g_{\mathbb{R}}} + v^2( \gamma_1(t)) | \dot{\gamma_2}(t)|^2_{g_L}} \]
Suppose that  $length_{g_L} (\gamma_2(t)) > d(\gamma_2(0), \gamma_2(1))$, and let $\sigma$ be a minimal geodesic in $L$ from $\gamma_2(0)$ to $\gamma_2(1)$.  Then 
\[ |\dot{\sigma}(t)| = length(\sigma) < length(\gamma_2) = \int_0^1 |\dot{\gamma_2}(t)| dt \]
In particular, there must be an open interval $(\alpha, \beta)$ with $|\dot{\sigma}(t)| <  |\dot{\gamma_2}(t)|$.  On $(\alpha, \beta)$ the curve $\overline{\gamma}(t) = (\gamma_1(t), \sigma(t))$ is clearly shorter than $\gamma|_{(\alpha, \beta)}$, which contradicts the fact that $\gamma$ is minimizing.

Now assume that $\gamma$ is a line.  From (1), in order to show that $\gamma_2$ is a line we just need to show that the length of both branches of $\gamma_2(s)$  as $s \rightarrow \infty$ and $s \rightarrow -\infty$ are infinite in $g_L$.  To see this we use the geodesic equations for the warped product, from which it follows (see \cite[Remark 39, p. 208]{O'Neill}) that the quantity $(v \circ \gamma_1)^4 g_L(\dot{\gamma_2}, \dot{\gamma_2}) = C$ for some constant $C$.  Since $v$ is bounded above, this implies that there is a universal constant $A$ not depending on $s$ such that  $g_L(\dot{\gamma_2}, \dot{\gamma_2})  \geq A$.  This implies that the length of both branches of $\gamma_2 $ in $L$ is infinite. 
\end{proof}

\begin{corollary} \label{Corollary:ProjectLines} For  the splitting given in Theorem  \ref{Theorem:IterativeSplitting},  any line in $(M,g)$  is constant on the $L$ factor. 
\end{corollary} 
On the other hand, for the metric \[  g = dr^2 + e^{\frac{2\phi(r)}{n-1}} g_{\mathbb{R}^{k-1}} + e^{\frac{2\phi(r)}{n-1}} g_{L},\]  the lines in the $\mathbb{R}^{k-1}$ factor will not necessarily lift to lines in $M$.  However, we can avoid this issue if we assume a two-sided bound on $f$, which will always be satisfied for the universal cover of a compact $CD(0,1)$ space. 

\begin{lemma} \label{Lemma:IterativeSplittingBounded}
Suppose that $(M,g,f)$ is $CD(0,1)$ with   $f$ is bounded (above and below) and contains a line,  then either $\phi$ is constant in  Theorem \ref{Theorem:IterativeSplitting} or $M$ is diffeomorphic to $\mathbb{R} \times L$ and  $g= dr^2 + e^{\frac{2\phi}{n-1}}g_L$ where $f= \phi + f_L$  and  $(L, g_L, f_L)$   has $(\mathrm{Ric}_{g_L})_{f_L}^0 > 0$. In particular, $(L, g_L)$ does not admit a line. 
\end{lemma}

\begin{proof} Split $ g = dr^2 + e^{\frac{2\phi}{n-1}}g_{L'}$, $f = \phi + f_{L'}$ as in the proof of  Theorem \ref{Theorem:IterativeSplitting}.  We claim if $\phi$ is non-constant then  $(\mathrm{Ric}_{g_{L'}})_{f_{L'}}^0 > 0$.  If $(\mathrm{Ric}_{g_{L'}})_{f_{L'}}^0(V,V)$ was not positive for some choice of $V$, then by Proposition \ref{Proposition:GradientExample}, $\frac{\partial^2 \phi} {\partial r^2} \leq 0$.  Since $f$ is a bounded function this implies that $\phi$ is bounded and concave function of $r$, so it  must be constant.
\end{proof}
Now we turn our attention to applications  of the splitting theorem to spaces with symmetry and the fundamental group.      These  come from studying  the isometry group of non-compact spaces which are $CD(0,1)$ with $f$ bounded that admit a line.    

When $\phi$ is constant we are in the  case considered by Cheeger-Gromoll where we have a product metric $g = g_{\mathbb{R}^k} + g_L$ where $L$ admits no lines.  The main observation is that isometries $F$ of $g$ must take lines to lines.  This implies that $F$  preserves the distributions tangent to $\mathbb{R}^k$ and $L$ in $M$.  Thus, $F$ splits into $F= F_1 \times F_2$ where $F_1\in \mathrm{Isom}(\mathbb{R}^k)$ and $F_2 \in  \mathrm{Isom}(L, g_L)$. 

When $\phi$ is not constant we obtain a similar result.  By Lemma \ref{Lemma:IterativeSplittingBounded} and Corollary \ref{Corollary:ProjectLines} we have $g = dr^2 + e^{\frac{2\phi}{n-1}}g_L$ and the only line for the $g$ metric  is the one in the $r$-direction.  Thus, since isometries take lines to lines,  we also have that $F$ splits as $F_1 \times F_2$ where $F_1: \mathbb{R} \rightarrow \mathbb{R}$ and $F_2: L \rightarrow L$, moreover,  simple calculation shows that for any isometry of this form for a warped product  we must have  $F_1 \in \mathrm{Isom}(\mathbb{R})$  with  $\phi \circ F_1 = \phi $ and $F_2 \in  \mathrm{Isom}(L, g_L)$ (see Exercise 11 on page 214 of \cite{O'Neill}). 

Now we can apply these results to the universal cover of a compact space which is $CD(0,1)$. 

\begin{theorem} \label{Theorem:CompactUniversalCover} Let $(M,g,f)$ be compact and $CD(0,1)$, let $(\widetilde{M}, \widetilde{g}, \widetilde{f})$ be the universal cover of $M$ with the covering metric $\widetilde{g}$ and $\widetilde{f}$ the pullback of $f$ to $\widetilde{M}$.  Then either 
\begin{enumerate}
\item $\widetilde{M}$ is compact, 
\item $(\widetilde{M}, \widetilde{g})$ is isometric to a product of a flat metric on $\mathbb{R}^k$ and a compact manifold $L$. 
\item $\widetilde{M}$ is diffeomorphic to $\mathbb{R} \times L$ where $L$ is compact and  $ \tilde{g} = dr^2 + e^{\frac{2\phi}{n-1}}g_L$, $\tilde{f} = \phi + f_L$,  and  $(L, g_L, f_L)$  is $CD(K, 0)$ for some $K>0$. 
\end{enumerate}
\end{theorem}

Note that Case (3) can certainly occur, as a metric  of the form $ \tilde{g} = dr^2 + e^{\frac{2\phi}{n-1}}g_L$ with $\phi$ periodic and $f=\phi$  will cover a $CD(0,1)$ metric on $S^1 \times L$,

\begin{proof}
Assume (1) is not true so that $\widetilde{M}$ is non-compact.   A standard argument shows that $\widetilde{M}$ contains a line.  To see this take a ray $\gamma$ in $M$ and  let $t_i \rightarrow \infty$.  Then, since the deck transformations of $\widetilde{M}$ act by isometries of $\widetilde{g}$ there is a compact set $K$ (e.g. a fundamental domain) and a sequence of isometries $F_i$ such that $F_i(\gamma(t_i)) \in K$ for all $i$.  Let $p$ be a limit of a convergent subsequence of the $F_i(\gamma(t_i))$. For some further subsequence we also have $DF_i(\dot{\gamma}(t_i))$ converging to a unit vector $v \in T_pM$.  Let $\sigma$ be the geodesic with $\sigma(0) = p$ and $\dot{\sigma}(0) = v$ then, since the distance that a geodesic minimizes is continuous with respect to its initial conditions, $\sigma$ is a line. 

We can then split $ \tilde{g} = dr^2 + e^{\frac{2\phi}{n-1}}g_L$, $\tilde{f} = \phi + f_L$.  If $\phi$ is constant, then we can split $M$ into $\mathbb{R}^k \times L$ where $L$ contains no lines.  If $L$ is non-compact, then the argument above, using the fact that the isometries of $\mathbb{R}^k \times L$  must split, would produce a line in $L$, therefore $L$ must be compact in this case, and we obtain (2). 

 Now suppose that $\phi$ is not constant. We  need to show that $L$ is compact. By Lemma \ref{Lemma:IterativeSplittingBounded} $L$ does not contain any lines.  The idea is to assume that $L$ is non-compact and argue by contradiction that $L$ must then contain a line.  This is complicated by the fact that geodesics of $L$ do not necessarily lift to geodesics of $\widetilde{M}$, so we must use the geodesic equations of a warped product again. 

 Fix $p \in L$ and  let $x_j$ be a sequence going off to infinity in $L$.  Let $\gamma^{j}$ be a unit speed minimal geodesic in $\widetilde{M}$ from $(0,p)$ to $(0, x_j)$.  Write $\gamma^{j}(t) = (\gamma^{j}_1(t), \gamma^{j}_2(t))$, using the warped product geodesic equations again  as above we see there is a  constant $C_j$ such that $e^{\frac{2\phi(\gamma^{j}_1(t))}{n-1}} |\dot{\gamma}^{j}_2(t) |_{g_L} = C_j $.   Since $\gamma_1 :(a,b) \rightarrow \mathbb{R}$ must have $\gamma_1(a) = \gamma_1(b) = 0$, it must have a critical point,  $t_0$.  At that point, 
 \[ 1 = |\dot{\gamma}^j|_{g} = e^{\frac{\phi(\gamma^{j}_1(t_0))}{n-1} }|\dot{\gamma}^{j}_2(t_0)|_{g_L}\]
 so $C_j = e^{\frac{\phi(\gamma^{j}_1(t_0))}{n-1}}$.  Since $\phi$ is bounded, this implies that $C_j$ is bounded.  Then  there is a positive constant  $A$ such that $|\dot{\gamma}^{j}_2(t) |_{g_L} \geq A$ for all $j$ and $t$.  
 
 Now consider $\gamma$ a ray which is a sub-sequential limit of $\gamma^{j}$.  Write $\gamma(t) = (\gamma_1(t), \gamma_2(t))$.   Since $|\dot{\gamma}^{j}_2(t) |_{g_L} \geq A$ we have that $|\dot{\gamma}_2(0) |_{g_L} \geq A$.    Using the geodesic equations for a warped product in the same way as above,  we obtain a possibly different $A$ such that $|\dot{\gamma}_2(t) |_{g_L} \geq A$ for all $t$.  Now take a sequence $\gamma(t_i)$ with $t_i \rightarrow \infty$ and pull back $\gamma$ by isometries $F^i$ to produce a line $\sigma$ as before.  Since each $F^i$ splits as a map $F^i_1 \times F^i_2$ where $F^i_2$ is an isometry of $L$, if we write $\sigma(s) = (\sigma_1(s), \sigma_2(s))$ then $|\dot{\sigma}_2(s) |_{g_L} \geq A$ for all $s$.  Therefore $\sigma_2$ is not a constant map so by Proposition \ref {Proposition:WPLines} the image of $\sigma_2$ forms a line in $L$  which achieves the desired contradiction. 
\end{proof}
We now use  Theorem \ref{Theorem:CompactUniversalCover} to prove Theorem \ref{Thm:FundGroup}. 

\begin{proof}[Proof of Theorem \ref{Thm:FundGroup}] 
First we show (3).  From the proof of Theorem \ref{Theorem:Splitting} using Lemma \ref{Lemma:Rigidity},  if $\widetilde{M}$ contains a line, then at every point $p$  there is a vector  $V \in T_pM$ with $\mathrm{Ric}_f^1(V,V) = 0$.  Thus by Theorem \ref{Theorem:CompactUniversalCover},  if $\mathrm{Ric}_f^1>0$ at a point, then $\widetilde{M}$ must be compact and $\pi_1(M)$ is finite. 

 Identify $\pi_1(M)$ as a subgroup of the isometries of $\widetilde{M}$ acting properly discontinuously and freely on $\widetilde{M}$.  Then as we discuss above, for  $F\in \pi_1(M)$ we can write $F = F_1\times  F_2$ where $F_2 \in \mathrm{Isom}(L, g_L)$ and, by Theorem \ref{Theorem:CompactUniversalCover}, $F_1$ is in the isometry group of flat $\mathbb{R}^k$ ($k=1$ when $\phi$ is non-constant).  The projection of $\pi_1(M)$ into each factor then produces a short  exact sequence  
\[ 0 \rightarrow E \rightarrow \pi_1(M) \rightarrow \Gamma  \rightarrow 0\]
Where $\Gamma$ is a crystallographic group, i.e. a discrete, cocompact subgroup of the isometry group of $\mathbb{R}^k$ and $E$ is a finite group. By \cite[Theorem 2.1]{Wilking} $\pi_1(M)$ is then the fundamental group of a compact manifold of nonnegative sectional curvature.  

Finally if $b_1(M) = n$ then $k$ must be $n$ and then $\widetilde{M}$ must be flat, implying that $M$ is also flat. 
\end{proof}
Finally we show that $\phi$ must be constant when $M$ is locally homogeneous. 

\begin{proof}[Proof of Theorem \ref{Theorem:LocallyHomogeneous}]

Let $\widetilde{M}$ be the universal cover of a compact  locally homogeneous space $M$.     Then $\widetilde{M}$ is homogeneous.  Apply Theorem \ref{Theorem:CompactUniversalCover} and suppose that $\phi$ were not constant. Then,  since the  isometry group splits and acts transitively, between any two points in $\mathbb{R}$  there must be a reflection or translation $F_1$  such that $\phi =  \phi \circ F_1$.   This implies that $\phi$ must be constant.  Then from Theorem  \ref{Theorem:CompactUniversalCover} $\widetilde{M} = \mathbb{R}^k \times L$ where $L$ is compact and homogeneous.  However, by  \cite[Proposition 3.7]{KennardWylie},  $L$ must then also  have non-negative Ricci curvature. Then $M$ has non-negative Ricci curvature.  The rest of the structure then follows  from \cite[Theorem 5]{CheegerGromoll}. 
\end{proof}

\section{Manifolds with boundary}
In this section, we prove  a version of the splitting theorem for compact manifolds with boundary. The boundary $\partial M$ is also assumed to be smooth with outward unit normal $\nu$.   Let $H$ be the mean curvature of $\partial M$ with respect to the outward normal vector.  The weighted (or generalized) mean curvature of the boundary is $H_f = H - g(\nabla f, \nu)$.  Just as the usual mean curvature arises in the first variation of the Riemannian volume, the weighted mean curvature arises in the first variation of the measure $e^{-f} d\mathrm{vol}_g$.

We have the following splitting phenomenon. 

\begin{theorem} \label{Theorem:Boundary} Suppose that $(M,g,f)$  is a compact manifold with boundary which is $CD(0,1)$, if $H_f \geq 0$ ($M$ is generalized mean convex) and $M$ has more than  one boundary component,  then $M$ is a warped product over an interval $M = [a,b] \times L$, $f = \phi(r) + f_L(x)$, where $\phi: [a,b] \rightarrow \mathbb{R}$ and $f_L: L \rightarrow \mathbb{R}$, and $g_M = dr^2 + e^{\frac{2\phi(r)}{n-1}}g_{L}$ for a fixed metric $g_L$ on $L$.   
\end{theorem}
The warped products in the  conclusion of the theorem have $H_f \equiv 0$ on $\partial M$ and $\mathrm{Ric}_f^1\left( \frac{\partial}{\partial r}, \frac{\partial}{\partial r} \right) = 0$, so we have the following corollary.

\begin{corollary}  Suppose that $(M,g,f)$  is a compact manifold with boundary which is $CD(0,1)$ and $M$ is generalized mean convex.    If $\mathrm{Ric}_f >0$ at a point in the interior of $M$, or $H_f >0$ at a point in $\partial M$, then $M$ has only one boundary component. 
\end{corollary}

By the same argument as in Corollary \ref{Corollary:Splitting} we also have an isometric product in the conclusion of Theorem \ref{Theorem:Boundary} if $(M,g,f)$ is $CD(0,N)$ for $N<1$. Using the ideas in the next section, Theorem \ref{Theorem:Boundary} can also  be extended to non-gradient fields, we leave the statement to the interested reader. 

\begin{proof}[Proof of Theorem \ref{Theorem:Boundary}.]  
Let $L$ be a boundary component of $M$ and let $r$ be the distance to $L$.  Let $\gamma$ be a unit speed geodesic from $L$ to a point $x \in M$ which minimizes the distance from $x$ to $L$ and such that $\gamma(t)$, $t>0$ is contained in the interior of $M$.  Then applying Lemma  \ref{Lem:Bochner} to $r$ gives 
\[    D_{\nabla r} (v^2 \Delta_f r(\gamma(r))) \leq 0
\]
Moreover, we have that $ \Delta_f r (\gamma(r)) \rightarrow - H_f(\gamma(0))$ as $r \rightarrow 0$, so we have that  $\Delta_f r\leq 0$ along the geodesic $\gamma$.

Now let $L_1$ be a component of $\partial M$.  Let $L_2$ be another boundary component which minimizes the distance from $L_1$ to $L_2$ among the other boundary components.     let $r_1, r_2 $ be the distance functions to $L_1$ and $L_2$ respectively.  Consider the function $e(x) = r_1(x) + r_2(x)$. By the triangle inequality $e(x) \geq d(L_1, L_2)$ and the points where the minimum is achieved must lie on a geodesic $\gamma$ which connects $L_1$ and $L_2$ and only touches $\partial M$ at its endpoints.  By the argument above  $\Delta_f e = \Delta_f r_1 + \Delta_f r_2  \leq 0$ at such a minimal point.  This implies that  $e$ must be constant by the strong maximum principle. Then there is a constant $a$ so that $r_1 = a - r_2$, which implies that  $\Delta_f r_1 = 0$. By elliptic regularity this shows that $r_1$ is smooth on the interior of $M$.   Then $r_1$ is a smooth function with  $|\nabla r_1|^2=1$  and $\Delta_f(r_1) = 0$.  The argument in Lemma \ref{Lemma:Rigidity} then shows that $M$ is a warped product. 
\end{proof}

\section{Non-gradient Vector fields}

In this section we explain how the results above also have versions for non-gradient potential fields.  Curvature dimension inequalities have a well known definition for vector fields. 

\begin{definition}
Let $X$ be a vector field on a Riemannian metric $(M^n,g)$. The $N$-dimensional generalized Ricci tensor is 
\[ \mathrm{Ric}_X^N  = \mathrm{Ric} + \frac{1}{2} L_X g - \frac{X^{\sharp} \otimes X^{\sharp}}{N-n} \]
where $L_X g$ is the Lie derivative of $g$  with respect to  $X$ and $X^{\sharp}$ is the dual one form of $X$ coming from $g$.  We say that  $(M,g, X)$ is  $\mathrm{CD}(\lambda, N)$, ($\lambda \in \mathbb{R}, N \in (-\infty, \infty]$) if   $\mathrm{Ric}_X^N \geq \lambda$. 
\end{definition}

Note that with this definition $\mathrm{Ric}_f = \mathrm{Ric}_{\nabla f}$, so the results of this section should be viewed as generalizing our results in section 3 to non-gradient fields.  

All of our results in the gradient case involve bounds on the potential function $f$.  While there is no potential function for a non-gradient field, we can still make sense of bounds by integrating $X$ along geodesics.   Let $X$ be a vector field on a Riemannian manifold $(M,g)$.  Let $\gamma:(a, b) \rightarrow M$ be a geodesic that is parametrized by arc-length.  Define 
\[ f_{\gamma} (t) = \int_a^t  g(\dot{\gamma}(s), X(\gamma(s))) ds \]
$f_{\gamma}$ is a real valued function  on the interval $(a,b)$ with the property that  $\dot{f_{\gamma}}(t) =  g(\dot{\gamma}(t), X(\gamma(t)))$.  When $X = \nabla f$ is a gradient field then $f_{\gamma} = f(\gamma(t)) - f(\gamma(a))$, in the non-gradient case we think of $f_{\gamma}$ as being the anti-derivative of $X$ along the curve $\gamma$.  We now introduce the  condition we will need for results in this section.  

\begin{definition} Let $(M,g)$ be a smooth non-compact complete Riemannian manifold with a smooth vector field $X$.   Then we say $(M, g,X)$ is  $X$-\emph{complete} if for every point $y \in M$ 
\[ \displaystyle \limsup_{r \rightarrow \infty} \inf_{l (\gamma) = r} \left\{\int_0^r e^{\frac{-2f_{\gamma}(\gamma(s))}{n-1}} ds  \right\}  = \infty.\]
where the infimum is taken over all  minimizing unit speed geodesics $\gamma$ of the metric $g$ with $\gamma(0) = y$. If $X=\nabla f$ we say that $(M,g,f)$ is $f$-complete.
\end{definition}

In general, $f_{\gamma}$ depends on the parametrization of $\gamma$ only up to an additive constant, so the notion of $X$-completeness does not depend on the parametrization of the geodesic.  Also note that if a vector field $X$ has the property that  $f_{\gamma}$ is  bounded for all unit speed minimizing geodesics then it is $X$-complete.  However, even in the gradient case, $f$-completeness is a weaker condition than $f$ bounded above. 

One way to interpret $f$-completeness is that the quantity $\int_0^r e^{\frac{-2f_{\gamma}(\gamma(s))}{n-1}} ds$ is, up to a multiplicative factor,  the energy of the curve $\gamma$ in the conformal metric $e^{\frac{-2f}{(n-1)}} g$.  From this we can see that $f$-completeness implies that $e^{\frac{-2f}{(n-1)}} g$ is a  complete metric.  Alternately, $X$-completeness is equivalent to the completeness of a certain modified affine connection, see \cite{WylieYeroshkin} for more details.

Our most general splitting theorem is the following. 

\begin{theorem} \label{Theorem:SplittingNonGradient}  Let $(M,g)$ be a complete Riemannian metric supporting a vector field $X$ which is   $CD(0,1)$ and $X$-complete.  If $(M,g)$ admits a line then  $M$ is a twisted product metric on $\mathbb{R} \times L$.  If $X= \nabla f$ then $M$ is a warped product. 
\end{theorem}

On the other hand, it is easy to see from the formulas in Proposition \ref{Proposition:Computation} that we can not obtain a warped product splitting for non-gradient fields.  
 \begin{proposition}  
There are metrics of the form $dr^2 + e^{\frac{2\phi}{n-1}} g_{S^n}$ which are $CD(0,1)$ where $\phi$ is not a function of $r$ and $X$ is not gradient.
\end{proposition}

\begin{proof}
For any function $\phi$,  let $X = \frac{2}{n-1} \frac{\partial \phi}{\partial r} + \left( \frac{n-3}{n-1} \right) \nabla \phi$.  Note that $X$ is a gradient field if and only if $\phi$ is a function of $r$.  Then a calculation using Proposition \ref{Proposition:Computation}  shows that $\mathrm{Ric}_X^1 \left( \frac{\partial}{\partial r}, Y \right) = 0$ for all $Y$.  The formula for $\mathrm{Ric}_X^1$ for vectors tangent to $S^{n}$ is much more complicated.  However, it is of the form 
\[ \mathrm{Ric}_X^1(U,V) = \mathrm{Ric}^{S^n}(U,V) + \text{ terms involving $\phi$ and its first and second partial derivatives. } \]

The terms  on the right will also go to zero as $\phi$ and its partial derivatives go to zero.  Therefore, if we take $g_{S^n}$ to be a round sphere with positive Einstein constant $\lambda$, there is a constant $A$ which depends on $\lambda$ and the dimension such that if  $\phi$ and its first and second derivatives are all less than $A$, then $\mathrm{Ric}_X^1 (U,V) \geq 0$.   
\end{proof}
Now we turn our attention to proving the splitting theorem.  The first component is the Bochner formula applied to the twisted Laplacian $\Delta_X = \Delta - D_X$, of the distance function, which follows from the same argument as in Lemma \ref{Lem:Bochner}. 

\begin{lemma}  \label{Lemma:BochnerDistanceFunction} Suppose $(M^n,g,X)$ is CD(0,1) and that $r$ is a smooth distance function on an open subset of a Riemannian manifold $(M,g)$.  Let $\gamma$ be an integral curve of $r$ and let $v_{\gamma} = e^{\frac{f_{\gamma}}{n-1}}$.  Then, 
\[ \frac{d}{dr}\left( v_{\gamma}^2 \Delta_X r \right)  \leq  - v_{\gamma}^2 \frac{(\Delta_X r)^2}{n-1}  \]
where $\frac{d}{dr}$ denotes the derivative along $\gamma$.  Moreover, if equality  is achieved at a point $p$  then the $(n-1)$ non-zero eigenvalues of $\mathrm{Hess} r|_p$ are all equal and $\mathrm{Ric}_X^{1}(\nabla r, \nabla r) = 0$ at $p$.
\end{lemma}

\begin{proof} As is well known,  the usual Bochner formula for functions, 
\[ \frac{1}{2} \Delta|\nabla h|^2  = |\mathrm{Hess} h|^2 + \mathrm{Ric}(\nabla h, \nabla h) + g(\nabla h, \nabla \Delta h), \]
can also be modified for non-gradient fields in the same manner as in (\ref{eqn:WeightedBochner}) to  
\begin{eqnarray*} \frac{1}{2}\Delta_X |\nabla  h|^2 =  |\mathrm{Hess} h|^2 + \mathrm{Ric}^{\infty}_X(\nabla h, \nabla h) + g(\nabla h, \nabla \Delta_X h).  \end{eqnarray*}
This follows directly from the identity
\begin{eqnarray*}
g(\nabla h, \nabla (D_X h))   &=&   D_{\nabla h} g(X, \nabla h) \\
&=& g(\nabla_{\nabla h} X, \nabla h) + g(X, \nabla_{\nabla h} \nabla h)  \\
&=&  \frac{1}{2} \left( L_X g(\nabla h, \nabla h) + D_X |\nabla h|^2 \right) 
\end{eqnarray*}
 The proof  is then identical to the proof of  Lemma \ref{Lem:Bochner} using $v_{\gamma}$ in the place of $v$ in the argument. 
\end{proof}
  
Following the same arguments as in Section 3, it then follows that if $(M,g,X)$ is $CD(0,1)$ and $X$-complete then  $\Delta_X(b_{\gamma}) \geq 0$ for any Busemann function.  In the gradient case, we then have the splitting theorem when $(M,g,f)$ is $f$-complete.   The other element needed for the splitting theorem  in the non-gradient case is a generalization of  Lemma \ref{Lemma:Rigidity} to the non-gradient case. 

 \begin{lemma} \label{Lemma:RigidityNonGradient} Suppose that  $(M,g)$ is a complete Riemannian manifold with a smooth vector field $X$ that is $CD(0,1)$.   If  there is a smooth function $r$ on $(M,g)$  such that $|\nabla r|^2=1$  and $\Delta_X r = 0$, then
\begin{enumerate}
\item  $M$ splits topologically as $\mathbb{R} \times N$, with metric of the form $g = dr^2 + e^{\frac{2\phi}{n-1}} g_N$, where $g_N$ is a metric on $N$ and $\phi:M \rightarrow \mathbb{R}$.
\item $\mathrm{Ric}_X^1(\nabla r, \nabla r) = 0$.  
\item $X = \frac{\partial \phi}{\partial r} \frac{\partial}{\partial r} + U$ where $U \perp \frac{\partial}{\partial r}$. 
\end{enumerate}
\end{lemma}

\begin{proof}
Since $|\nabla r|=1$ we have $\mathbb{R} \times N$ topologically and  $g = dr^2 + g_r$, where $g_r$ is the metric restricted to a  level set of $r$.  In terms of this splitting write the vector field $X = a(r, x) \frac{\partial}{\partial r} + Y$  where  $a:M \rightarrow \mathbb{R}$ and $Y$ is tangent to $N$ at every point.  Then,  for $\gamma$ an integral curve of $r$,  we have $g(X, \dot{\gamma}) = a$.  Define a function $\phi$ globally on $M$ via the formula 
\[ \phi(r,x) = \int_0^r  a(t, x) dt.  \]
$\phi$ is  clearly a smooth function since $a$ is smooth.    

The assumptions imply that we have equality in Lemma \ref{Lemma:BochnerDistanceFunction}, so $\mathrm{Ric}_X^{1} (\nabla r, \nabla r)= 0$ and $\mathrm{Hess} r = \alpha g_r$ for some function $\alpha$.  But we also have $\Delta r = (n-1) \alpha = g(X, \nabla r)$ so 
\[ \mathrm{Hess} r =\frac{ g(X, \nabla r)}{n-1} g_r \]
Since $D_{\nabla r} \phi = a(r,x) = g(X, \nabla r)$ this implies that 
\[ L_{\nabla r} \left(e^{\frac{-2\phi}{n-1}} g_r\right) = 0 \]
which implies that 
$g_r = e^{\frac{2(\phi(r, \cdot)- \phi(0, \cdot))}{n-1}} g_0$. 
This gives us that the metric is a twisted product $g = dr^2 + e^{\frac{2\phi}{n-1}} g_N$ where $g_N=e^{\frac{-2 \phi(0, \cdot)}{n-1}} g_0$ is a fixed metric on $N$.    Note that the function $\phi$ automatically  satisfies (3)  as $\frac{\partial \phi}{\partial r}= a(r,x)$.
\end{proof}

The proof of Theorem \ref{Theorem:SplittingNonGradient} then follows using Lemma \ref{Lemma:RigidityNonGradient}  and  the same arguments as in Section 3. In the $CD(0,N)$ case $N<1$ we also obtain the isometric product splitting. 

\begin{corollary} 
Suppose that $(M,g)$ is a complete Riemannian manifold and $X$ is a smooth vector field on $M$ which is $X$-complete and $CD(0,N)$ for $N<1$ .  If $(M,g)$ admits a line then $M$ is isometric to a product metric $M = \mathbb{R} \times L$ and $X$ is a vector field on $L$.
\end{corollary}

\begin{proof}
Since $(M,g,X)$ is $CD(0,N)$, it is also $CD(0,1)$ so Theorem \ref{Theorem:SplittingNonGradient} implies that $g$ is a  twisted  product, $g= dr^2 + e^{\frac{2\phi}{n-1}} g_L$ .    We also have that $\mathrm{Ric}_X^{1}( \frac{\partial}{\partial r} ,  \frac{\partial}{\partial r} ) = 0$. Since $X =   \frac{\partial \phi}{\partial r} \frac{\partial}{\partial r} + U$ where $U \perp \frac{\partial}{\partial r}$, this gives us 
\begin{eqnarray*}
\mathrm{Ric}_X^{L}\left( \frac{\partial}{\partial r} ,  \frac{\partial}{\partial r} \right) &=&  -\left( \frac{1-N}{(n-1)(n-N)} \right) \left(\frac{\partial \phi}{ \partial r}\right)^2 
\end{eqnarray*}
so we must have $\frac{\partial \phi}{\partial r} = 0$.    This implies that metric $g$ is a product metric, which we can write as $g = dr^2 + h_L$, where $h_L$ is a conformal metric to $g_L$.  

We can also show $X$ is a vector field on $L$ only using the fact that $\mathrm{Ric}_X^1\left( \frac{\partial}{\partial r}, V\right) = 0$ for all $V \perp \frac{\partial}{\partial r}$. To see this, fix a point $x$ in $N$, let $\frac{\partial}{\partial y^i}$, $i=1, \dots, n-1$ be an orthonormal basis of local coordinates around $x$ in the $g_L$ metric.  Write $X =  b_i \frac{\partial}{\partial y^i}$  Then 
\begin{eqnarray*}
0= \mathrm{Ric}_X^1\left( \frac{\partial}{\partial r}, \frac{\partial}{\partial y^k}\right) = \frac{1}{2} \frac{\partial b_k}{\partial r} 
\end{eqnarray*}
So $X$ is a vector field on $L$ that does not depend on $r$.
\end{proof}

We would also like to generalize Theorem \ref{Theorem:IterativeSplitting} to the case where $X= \nabla f$, but the upper bound on $f$ is replaced by $f$-completeness.  The only obstacle that arises in the proof is  that when a space splits, it is not a priori clear that $f$-completeness on the whole space should imply $f$-completeness on the fiber.  It turns out, however,  that we can use  geodesic equations for a warped product to show that $f$-completeness does have this natural property for the spaces in our splitting theorem.   

Consider a split space to be a warped product of the form $ dr^2 + e^{\frac{2\phi}{n-1}} g_L$ with potential function $f = \phi(r) + f_L$.  Let $\gamma$ be a unit speed minimizing geodesic and write  $\gamma(s) = (\gamma_1(s), \gamma_2(s))$,  where $\gamma_1$ and $\gamma_2$ are the projections in the factors $\mathbb{R}$ and $L$.  The $f$-completeness condition implies for any ray of $(M,g_M)$  that $ \int_0^{\infty} \left( e^{\frac{-2\phi(\gamma_1(s))}{n-1} }  e^{\frac{-2f_L(\gamma_2(s))}{n-1} }  \right) ds $ diverges.  We have the following proposition. 

\begin{proposition} \label{Proposition:WPLinesBED}
Suppose $M$ is a split space that is $f$-complete and let $\gamma: (a,b) \rightarrow M$ be a minimizing geodesic of the form $\gamma(s) = (\gamma_1(s), \gamma_2(s))$ then 
\begin{enumerate}
\item If $\gamma_2$ is not constant and $\gamma$ is a line in $M$, then the image of $\gamma_2$ is a line in $L$. 
\item  The manifold with density $(L, g_L, f_L)$  is $f_L$-complete. 
\end{enumerate}
\end{proposition}

\begin{proof} 
Assume that $\gamma$ is a line.  From  part  (1) of Proposition \ref{Proposition:WPLines}, which is true for any warped product, in order to show that $\gamma_2$ is a line we just need to show that the length of both branches of $\gamma_2(s)$  as $s \rightarrow \infty$ and $s \rightarrow -\infty$ are infinite in $g_L$.  From the geodesic equations for the warped product, we have $e^{\frac{4\phi(\gamma_1(s))}{n-1} } g_L(\dot{\gamma_2}, \dot{\gamma_2}) = C$ for some constant $C$.  Then 
\begin{eqnarray*}
length(\gamma_2|_{[0, \infty)}) &=& \int_0^{\infty} |\dot{\gamma_2}|_{g_L} ds = C \int_0^{\infty} e^{\frac{-2\phi(\gamma_1(s))}{n-1} }ds. 
\end{eqnarray*}

Assume for contradiction that  $\gamma_2$ had finite length in $L$.  Then the function $ e^{\frac{-2f_L(\gamma_2(s))}{n-1}} $ is uniformly  bounded in $s$  and so  the $f$-completeness assumption applied to $\gamma$ implies that $\int_0^{\infty} e^{\frac{-2\phi(\gamma_1(s))}{n-1} } ds$ is infinite.  Up to a constant this is the length of $\gamma_2$, so we obtain a contradiction.  The same argument also shows that the length of the branch of $\gamma_2$ with $s \rightarrow -\infty$ is also infinite. 

In order to show (2), fix a point $p \in L$ and let $\beta(\tau)$ be a unit speed geodesic in $L$ with $\beta(0) = p$ which is minimizing for $\tau \in [0,r]$.  We want to estimate  $\int_0^r e^{\frac{-2f_L(\beta(\tau))}{n-1} }  d\tau$.  First note that from part (1) of Proposition \ref{Proposition:WPLines} and the uniqueness of minimizing geodesics  that there is a geodesic in $M$ which is of the form $\gamma(s) = (\gamma_1(s), \gamma_2(s))$  $s \in (0, t)$ such that the image of $\gamma_2$ is the image of $\beta$.   

We again apply the warped product geodesic equations to the geodesic $\gamma$ to see that there is a constant $C_{\gamma}  \neq 0$ such that  $ |\dot{\gamma_2}|_{g_L} =  C_{\gamma} e^{\frac{-2\phi(\gamma_1(s))}{n-1} } $.  By compactness, we can thus choose $C_1, C_2 $ uniformly so that   $C_1e^{\frac{-2\phi(\gamma_1(s))}{n-1} }  \leq   |\dot{\gamma_2}|_{g_L} \leq  C_2 e^{\frac{-2\phi(\gamma_1(s))}{n-1} } $ for every unit speed geodesic $\beta$ with $\beta(0) = p$.   
\begin{eqnarray}
\nonumber \int_0^{t} \left(e^{\frac{-2\phi(\gamma_1(s))}{n-1} }  e^{\frac{-2f_L(\gamma_2(s))}{n-1} } \right) ds &\leq & \frac{1}{C_1} \int_0^{t} e^{\frac{-2f_L(\gamma_2(s))}{n-1} }  |\dot{\gamma_2}|_{g_L} ds\\
 &=&\frac{1}{C_1}  \int_0^r e^{\frac{-2f_L(\beta(\tau))}{n-1} }  d\tau \label{LastEqn1}
 \end{eqnarray}
 where in the last line, we have re-parametrized the curve $\gamma_2$ by arc-length in $g_L$ to obtain $\beta$.    Moreover, 
 \begin{eqnarray}
r = length(\beta)  =   length(\gamma_2) \leq  C_2 \int_0^{t} e^{\frac{-2\phi(\gamma_1(s))}{n-1} }  ds  \label{LastEqn2}
\end{eqnarray}
where $t$ is the length in $M$ of the geodesic $\gamma$. 

(\ref{LastEqn2}) implies that as $r \rightarrow \infty$,  $t\rightarrow \infty$.  Then $f$-completeness of $M$ implies that the left hand side of (\ref{LastEqn1}) blows up as $r \rightarrow \infty$ and thus so does the right hand side,  showing that  $L$ is $f_L$-complete. 
  \end{proof} 
  
  On the other hand, in the non-gradient setting, we do not have the applications of the splitting theorem to the universal cover of a compact manifold because it is not true that a vector field lifted to the universal cover is $X$-complete. The question of the extent to which the Theorems \ref{Thm:FundGroup} and \ref{Theorem:LocallyHomogeneous} are true for non-gradient fields appears to be largely open.

\end{document}